\documentclass[10pt,a4paper]{article}

\usepackage{amssymb}
\usepackage{amsthm}
\usepackage{amsmath}

\newtheorem{thm}{Theorem}
\newtheorem{lem}{Lemma}

\newcommand*{\ev}{\ensuremath{\mathbf{E}}}
\newcommand*{\law}{\ensuremath{\overset{\mathsf{d}}{=}}}
\newcommand*{\pr}{\ensuremath{\mathbf{P}}}
\newcommand*{\xis}{\ensuremath{1 + \xi_1 + \xi_1 \xi_2 + \cdots}}
\newcommand*{\acoef}[1]{\ensuremath{a_{j_1 \dots j_{#1-1}}^{(#1)}}}

\begin{document}


\title{Moments of an exponential functional of random walks and
permutations with given descent sets}

\author{Tam\'as Szabados\footnote{Corresponding author, address:
Department of Mathematics, Budapest University of Technology and
Economics, M\H{u}egyetem rkp 3, H \'ep V em, Budapest, 1521,
Hungary, e-mail: szabados@math.bme.hu} \footnote{Research
supported by the French--Hungarian intergovernmental grant
``Balaton'' F-39/2000.} and Bal\'azs Sz\'ekely\footnote{Research
supported by the HSN laboratory of BUTE.} \\
Budapest University of Technology and Economics}

\date{}

\maketitle

\begin{center}
Running head: \emph{Moments and permutations}
\end{center}

\begin{abstract}
The exponential functional of simple, symmetric random walks with
negative drift is an infinite polynomial $Y = 1 + \xi_1 + \xi_1
\xi_2 + \xi_1 \xi_2 \xi_3 + \cdots$ of independent and identically
distributed non-negative random variables. It has moments that are
rational functions of the variables $\mu_k = \ev(\xi^k) < 1$ with
universal coefficients. It turns out that such a coefficient is
equal to the number of permutations with descent set defined by
the multiindex of the coefficient. A recursion enumerates all
numbers of permutations with given descent sets in the form of a
Pascal-type triangle.
\end{abstract}

\renewcommand{\thefootnote}{\alph{footnote}}
\footnotetext{ 2000 \emph{MSC.} Primary 05A05, 05A19. Secondary
60C05.} 

\section{Introduction}

The present work was induced by the following problem from
probability theory, cf. \cite{expfun}. Let $(X_j)_{j=1}^{\infty}$
be a sequence of independent and identically distributed random
variables with $\pr(X_j = \pm 1) = \frac12$. Further, let $S_0=0$,
$S_k = \sum_{j=1}^k X_j$ $(k \ge 1)$ be a simple, symmetric random
walk. Introduce the following exponential functional $Y$ of the
random walk with negative drift:
\begin{equation}\label{eq:Y}
Y = \sum_{k=0}^{\infty} \exp(S_k - k \nu ) = \xis , \qquad
\xi_j = \exp(X_j - \nu) ,
\end{equation}
where $\nu > 0$. $Y$ is an approximation of a widely investigated
exponential functional of Brownian motion, important for example
in studying Asian options of financial mathematics. To generalize
it somewhat, let $(\xi_j)_{j=1}^{\infty}$ be independent and
identically distributed random variables, $\xi_j \ge 0$. Consider
first the finite polynomials
\begin{eqnarray*}
Y_n &=& 1 + \xi_1 + \xi_1 \xi_2 + \cdots + \xi_1 \cdots \xi_n\\
&=& 1 + \xi_1(1 + \xi_2 + \xi_2 \xi_3 + \cdots + \xi_2 \cdots \xi_n)
\qquad (n \ge 1) .
\end{eqnarray*}
This equation implies the following equality in distribution (denoted
by $\law$): $Y_n \law 1 + \xi Y_{n-1}$, where $\xi \law \xi_1$ and $\xi$
is independent of $Y_{n-1}$. Since
$Y_n \nearrow Y = 1 + \xi_1 + \xi_1 \xi_2 + \xi_1 \xi_2 \xi_3 +
\cdots$ as $n \to \infty$, in the limit we get a stochastic difference
equation for the infinite polynomial $Y$:
\[
Y \law 1 + \xi Y ,
\]
where $\xi$ is independent of $Y$.

Then using the binomial theorem, the following recursion
can be obtained for the $p$th moment $e_p = \ev(Y^p)$ of $Y$ for any
positive integer $p$:
\begin{equation}\label{eq:moments}
e_p = \frac{1}{1-\mu_p} \sum_{k=0}^{p-1} \binom{p}{k} \mu_k \: e_k ,
\end{equation}
supposing $\mu_p < 1$, where $\mu_k = \ev(\xi^k)$, $e_k
= \ev(Y^k)$, $k \ge 0$. We mention that $\ev(Y^p) < \infty$ if
and only if $\mu_p < 1$, and then $\mu_k < 1$ for any $1 \le k < p$
as well, cf. \cite[Theorem 2]{expfun}. Observe that
(\ref{eq:moments}) defines a recursive sequence $e_p = e_p(\mu_1,
\dots, \mu_p)$ for $p \ge 1$ with $e_0 =
1$, irrespective of any probability theory background. In this
recursion the $\mu_k$'s may be considered as variables that may take
any value except 1 for $k \ge 1$. Thus from now on we always suppose
that $\mu_k \ne 1$ for $k\ge 1$ and $\mu_0=1$.

It follows from (\ref{eq:moments}) by induction
that $e_p$ is a rational function of the variables $\mu_1, \dots,
\mu_p$ for any integer $p \ge 1$:
\begin{equation} \label{eq:fraction}
e_p = \frac{1}{(1-\mu_1) \cdots (1-\mu_p)} \sum_{(j_1, \dots,
j_{p-1}) \in \{0, 1\}^{p-1}} \acoef{p} \mu_1^{j_1} \cdots
\mu_{p-1}^{j_{p-1}} ,
\end{equation}
where the coefficients of the numerator are \emph{universal} constants,
being independent of the parameters in the recursion.

These coefficients $\acoef{p}$ make a symmetrical, Pascal's
triangle-like table if each row is listed in the increasing order
of the binary numbers $j_{p-1}2^{p-2} + \cdots + j_1 2^0$, defined
by the multiindices $(j_1, \dots , j_{p-1})$, see the rows $p = 1,
\dots , 5$ in Table 1. 
\begin{table}[h]
\caption{The Pascal's triangle-like table of the coefficients}
\[
\begin{array}{ccccccccccccccccc}
&&&&&&&&1&&&&&&&& \\
&&&&&&&\overset{0}{1}&&\overset{1}{1}&&&&&&& \\
&&&&&&\overset{00}{1}&\overset{01}{2}&&\overset{10}{2}
&\overset{11}{1}&&&&&& \\
&&&&\overset{000}{1}&\overset{001}{3}&\overset{010}{5}
&\overset{011}{3}&&\overset{100}{3}&\overset{101}{5}
&\overset{110}{3}&\overset{111}{1}&&&& \\
\overset{0000}{1}&\overset{0001}{4}&\overset{0010}{9}
&\overset{0011}{6}&\overset{0100}{9}&\overset{0101}{16}
&\overset{0110}{11}&\overset{0111}{4}
&&\overset{1000}{4}&\overset{1001}{11}&\overset{1010}{16}
&\overset{1011}{9}&\overset{1100}{6}&\overset{1101}{9}
&\overset{1110}{4}&\overset{1111}{1} \\
\end{array}
\]
\end{table}

Two natural questions may arise at this point, independently of
any probability theory background mentioned above. First, suppose
that one defines a recursive sequence $(e_p)_{p=1}^{\infty}$ by
(\ref{eq:moments}) with coefficients $\acoef{p}$ given by
(\ref{eq:fraction}). Can one attach any \emph{direct} mathematical
meaning to these coefficients $\acoef{p}$ then? The answer is yes,
and rather surprisingly (as was conjectured in \cite{expfun}), the
coefficient $\acoef{p}$ is equal to the number of permutations
$\pi \in S_p$ having descent $\pi(i) > \pi(i+1)$ exactly where
$j_i=1$, $1 \le i \le p-1$, cf. Theorem \ref{th:descent} below.

Second, can one give a direct way to evaluate the coefficients
$\acoef{p}$? The affirmative answer to this question is partly
included in the previous answer, since several formulae are known for
the number of permutations with given descent sets. However, an
apparently new recursion was conjectured in \cite{expfun}, which is
analogous to the recursion of binomial coefficients in the ordinary
Pascal's triangle. The proof of this is the content of Lemma
\ref{le:newrec} below.

\section{The results}

In the next lemma we establish a direct recursion for the
coefficients $\acoef{p}$.

\begin{lem} \label{le:firstrec}
Fix a multiindex $(j_1, \dots ,j_{p-1}) \in \{0, 1\}^{p-1}$. Let $S$
be the set of indices $k$ where $j_k = 1$:
\begin{equation} \label{eq:descset}
S = \{s_1, \dots, s_m\} = \{k: j_k = 1, 1 \le k \le p-1\}, \quad
m = \sum_{k=1}^{p-1} j_k .
\end{equation}
Then the coefficient $\acoef{p}$ defined by (\ref{eq:fraction})
can be obtained by the recursion
\begin{eqnarray}
\acoef{p} &=& \sum_{k=0}^{p-1} \binom{p}{k} j_k (-1)^{j_{k+1} + \cdots
+ j_{p-1}} \acoef{k} \nonumber \\
&=&  \sum_{l = 0}^m \binom{p}{s_l} (-1)^{m-l} \acoef{s_l},
\label{eq:firstrec}
\end{eqnarray}
where, by definition, $a^{(0)} = 1$, $j_0 = 1$, $s_0 = 0$ and $-1$
powered to an empty sum is $1$.
\end{lem}

\begin{proof}
The second equality in (\ref{eq:firstrec}) is a direct consequence of
the definitions above. To show the first equality, substitute
(\ref{eq:fraction}) into (\ref{eq:moments}):
\[
e_p = \frac{1}{1-\mu_p} \sum_{k=0}^{p-1}
\binom{p}{k} \frac{\mu_k}{(1-\mu_1) \cdots (1-\mu_k)}
\sum_{(j_1, \dots, j_{k-1}) \in \{0,1\}^{k-1}} \acoef{k}
\mu_1^{j_1} \cdots \mu_{k-1}^{j_{k-1}} .
\]
Here, multiplying by the common denominator and then collecting the
coefficients of $\mu_1^{j_1} \cdots \mu_{p-1}^{j_{p-1}}$ for each
$(j_1, \dots, j_{p-1})$ we obtain
\begin{eqnarray*}
\lefteqn{e_p (1-\mu_1) \cdots (1-\mu_p)} \\
&=& \sum_{k=0}^{p-1} \sum_{(j_1, \dots,
j_{k-1}) \in \{0,1\}^{k-1}} \binom{p}{k} \acoef{k}
\mu_1^{j_1} \cdots \mu_{k-1}^{j_{k-1}} \mu_k
(1-\mu_{k+1}) \cdots (1-\mu_{p-1}) \\
&=& \sum_{(j_1, \dots, j_{p-1}) \in \{0,1\}^{p-1}} \mu_1^{j_1}
\cdots \mu_{p-1}^{j_{p-1}} \sum_{k=0}^{p-1} \binom{p}{k} \acoef{k}
j_k (-1)^{j_{k+1}} \cdots (-1)^{j_{p-1}} .
\end{eqnarray*}
This and (\ref{eq:fraction}) imply the first equality in
(\ref{eq:firstrec}).
\end{proof}

Now we turn to the proof of the equality of the coefficient
$\acoef{p}$ given by (\ref{eq:fraction}) and the number of
permutations $b^{(p)}(S)$ with descent set $S$ given by
(\ref{eq:descset}).  The descent set of a permutation $\pi \in S_p$
is defined as $D(\pi)=\{i : \pi(i) > \pi(i+1), 1 \le i \le p-1 \}$.
It is known, cf. \cite[p. 69]{Stanley}, that the number of
permutations $\pi \in S_p$ with a given descent set $S = (s_1,
\dots ,s_m)$, $1 \le s_1 < \cdots < s_m \le p-1$, can be obtained by
the following inclusion-exclusion formula:
\begin{eqnarray}
b^{(p)}(S) &=& b^{(p)}(s_1, \dots, s_m) \nonumber \\
&=& \sum_{1 \le i_1 < \cdots
< i_j \le m} (-1)^{m-j} \binom{p}{s_{i_1}, s_{i_2}-s_{i_1}, \dots ,
s_{i_j}-s_{i_{j-1}}, p-s_{i_j}} . \label{eq:sieve}
\end{eqnarray}

\begin{thm}\label{th:descent}
The coefficient $\acoef{p}$ given by (\ref{eq:moments}) and
(\ref{eq:fraction}) is equal to the number of permutations
$b^{(p)}(S)$ with descent set $S$ given by (\ref{eq:descset}).
\end{thm}

\begin{proof}
It is enough to show that the numbers $b^{(p)}(S)$ satisfy the same
recursion (\ref{eq:firstrec}) as the numbers $\acoef{p}$ do, that is,
\begin{equation} \label{eq:firstrecb}
b^{(p)}(s_1, \dots, s_m) = \sum_{l=0}^m \binom{p}{s_l} (-1)^{m-l}
b^{(s_l)}(s_1, \dots, s_{l-1}) ,
\end{equation}
where, by definition, $s_0 = 0$ and $b^{(s_0)}= b^{(0)} = 1$.

To show this, let us substitute the inclusion-exclusion formula
(\ref{eq:sieve}) into the right hand side of (\ref{eq:firstrecb}):
\begin{eqnarray*}
\lefteqn{\sum_{l=0}^m \binom{p}{s_l} (-1)^{m-l} b^{(s_l)}(s_1,
\dots, s_{l-1})} \\
&=& (-1)^m + \sum_{l=1}^m \binom{p}{s_l} (-1)^{m-l} \sum_{1 \le
i_1 < \cdots < i_j \le l-1} (-1)^{l-1-j} \binom{s_l}{s_{i_1}, s_{i_2}-s_{i_1},
\dots, s_l - s_{i_j}} \\
&=& (-1)^m + \sum_{l=1}^m \sum_{1 \le i_1 < \cdots < i_j < l}
(-1)^{m-j-1} \binom{p}{s_{i_1}, s_{i_2}-s_{i_1}, \dots, s_l-s_{i_j}, p-s_l} \\
&=& (-1)^m + \sum_{1 \le i_1 < \cdots < i_j < l \le m} (-1)^{m-(j+1)}
\binom{p}{s_{i_1}, s_{i_2}-s_{i_1}, \dots, s_l-s_{i_j}, p-s_l} \\
&=& \sum_{1 \le i_1 < \cdots < i_r \le m} (-1)^{m-r} \binom{p}{s_{i_1},
s_{i_2}-s_{i_1}, \dots, s_{i_r}-s_{i_{r-1}}, p-s_{i_r}}
= b^{(p)}(s_1, \dots , s_m) .
\end{eqnarray*}
This proves (\ref{eq:firstrecb}), and so completes the proof.
\end{proof}

Lemma \ref{le:firstrec} described a recursion that uses all
previous rows of Table 1 to compose the elements of a new row. In
the next lemma we show a recursion that uses only the previous row
and which is an analog of the recursion formula in the ordinary
Pascal's triangle: $\binom{p}{k} = \binom{p-1}{k-1} +
\binom{p-1}{k}$ . Interestingly, an application of this identity
is a key step in the following algebraic proof as well. We also
give a simple combinatorial proof which basically translates the
well-known method by which permutations in $S_p$ can be obtained
from permutations in $S_{p-1}$ by adjoining the number $p$.

\begin{lem} \label{le:newrec}
The following recursion holds for any $p \ge 2$ and multiindex
$(j_1, \dots, j_{p-1}) \in \{0, 1\}^{p-1}$:
\begin{equation} \label{eq:newrec}
\acoef{p} =  \sum_{i=1}^{p-1} \delta_i a^{(p-1)}_{j_1^{(i)}
\dots j_{p-2}^{(i)}}
= \sum_{(i_1, \dots, i_{p-2}) \in L(j_1, \dots, j_{p-1})} a_{i_1
\dots i_{p-2}}^{(p-1)},
\end{equation}
where $a^{(1)} = 1$, $\delta_i = |j_i - j_{i-1}|$ for $i \ge 2$,
$\delta_1 = 1$, $j_k^{(i)} = j_k$ for $1 \le k \le i-1$, $j_k^{(i)} =
j_{k+1}$ for $i \le k \le p-2$, and $L(j_1, \dots, j_{p-1})$ is the
set of all distinct binary sequences obtained from $(j_1, \dots,
j_{p-1})$ by deleting exactly one digit. For example, $a_{0110}^{(5)}
= 11 = a_{110}^{(4)} + a_{010}^{(4)} + a_{011}^{(4)}$.
\end{lem}

\begin{proof}
First we prove the second equality in (\ref{eq:newrec}). For this
it is enough to show that if the same binary sequence is obtained
from $(j_1, \dots, j_{p-1})$ when eliminating either the $k$th or
the $l$th digit ($k < l$), then all digits between the $k$th and
$l$th (including these two) are uniformly either $0$'s or $1$'s (a
run of $0$'s or $1$'s). Therefore, the two recursions given in
(\ref{eq:newrec}) are the same.

Consider a multiindex $(j_1, \dots, j_{k-1}, j_k, \dots, j_l,
j_{l+1}, \dots, j_{p-1}) \in \{0, 1\}^{p-1}$. Suppose that we get the
same binary sequence by deleting $j_k$ and $j_l$, respectively:
$(j_1, \dots, j_{k-1}, j_{k+1}, \dots, j_l, j_{l+1}, \dots, j_{p-1})
= (j_1, \dots, j_{k-1}, j_k, \dots, j_{l-1}, j_{l+1}, \dots,
j_{p-1})$.  Then $j_k = j_{k+1} = \cdots = j_{l-1} = j_l$, so the
second equality in (\ref{eq:firstrec}) really holds.

Now it remains to show the first equality in (\ref{eq:newrec}), that
is, the recursion itself.

\emph{A combinatorial proof of the recursion.} Given a binary
sequence $(j_1, \dots, j_{p-1})$, let us remove a single $1$ from
a run of $1$'s or a single $0$ from a run of $0$'s. Count the
number of permutations in $S_{p-1}$ determined by the resulting
multiindex $(i_1, \dots, i_{p-2})$. This number is $a_{i_1 \dots
i_{p-2}}^{(p-1)}$ by Theorem \ref{th:descent}. We want to show
that there is a uniquely determined adjoining of the number $p$ to
any such permutation from $S_{p-1}$ to obtain a  permutation from
$S_p$ corresponding to the original multindex $(j_1, \dots,
j_{p-1})$.

If a $0$ was deleted from a run of $0$'s, the number $p$ should be
inserted right after the number at the position of the first $1$
following the affected run of $0$'s. (If the given run happens to
be the last, $p$ is inserted as the last number.) When a $1$ was
deleted from a run of $1$'s, the number $p$ should be inserted
right after the number at the position of the last $0$ preceding
the affected run of $1$'s. (If the given run happens to be the
first, $p$ is inserted as the first number.) Since these
insertions are the only ones that reconstruct the original descent
set, the recursion is proved.

\emph{An algebraic proof of the recursion.}
We are going to proceed by induction over
$p$. Thus suppose that the recursion holds for all multiindices of
lengths smaller than $p-1$.  First we use the recursion of Lemma
\ref{le:firstrec} for the terms in the right side of
(\ref{eq:newrec}), then we change the order of summation to obtain
\begin{eqnarray*}
\sum_{i=1}^{p-1} \delta_i a^{(p-1)}_{j_1^{(i)} \dots
j_{p-2}^{(i)}}
&=& \sum_{i=1}^{p-1} \delta_i \sum_{k=0}^{p-2}
\binom{p-1}{k} j_k^{(i)} (-1)^{j_{k+1}^{(i)} + \cdots
+ j_{p-2}^{(i)}} a^{(k)}_{j_1^{(i)} \dots j_{k-1}^{(i)}} \\
&=& \sum_{k=0}^{p-2} \binom{p-1}{k} \left\{ j_{k+1} (-1)^{j_{k+2}
+ \cdots + j_{p-1}} \sum_{i=1}^k \delta_i a^{(k)}_{j_1^{(i)}
\dots j_{k-1}^{(i)}}  \right. \\
&& \left. + j_k (-1)^{j_{k+1} + \cdots + j_{p-1}} \acoef{k}
\sum_{i=k+1}^{p-1} \delta_i (-1)^{j_i} \right\} .
\end{eqnarray*}

Here in the last expression, one may use the induction hypothesis
for the first sum. In the second sum observe that $\delta_i
(-1)^{j_i}$ is $1$ if $j_{i-1} = 1$ and $j_i = 0$, it is $-1$ if
$j_{i-1} = 0$ and $j_i = 1$, and it equals $0$ otherwise. Hence we
get the identity $j_k \sum_{i=k+1}^{p-1} \delta_i (-1)^{j_i} = j_k
(1-j_{p-1})$. Thus one obtains that
\begin{eqnarray*}
\lefteqn{\sum_{i=1}^{p-1} \delta_i a^{(p-1)}_{j_1^{(i)} \dots
j_{p-2}^{(i)}}
= \sum_{k=0}^{p-2} \left\{ \binom{p-1}{k} j_{k+1} (-1)^{j_{k+2}
+ \cdots + j_{p-1}} a^{(k+1)}_{j_1 \dots j_k} \right.} \\
&& \qquad \qquad \qquad \left. + \binom{p-1}{k} (1-j_{p-1})
j_k (-1)^{j_{k+1} + \cdots + j_{p-1}} \acoef{k} \right\}  \\
&=& \sum_{k=1}^{p-2} \left\{ \binom{p-1}{k-1} + \binom{p-1}{k}
(1-j_{p-1}) \right\} j_k (-1)^{j_{k+1} + \cdots + j_{p-1}} \acoef{k} \\
&& + (1-j_{p-1}) (-1)^{j_1 + \cdots + j_{p-1}} + \binom{p-1}{p-2}
j_{p-1} a^{(p-1)}_{j_1 \dots j_{p-2}} .
\end{eqnarray*}
To rewrite the terms above we use recursion (\ref{eq:firstrec}) in
the following case:
\[
a^{(p-1)}_{j_1\dots j_{p-2}} = \sum_{k=0}^{p-2} \binom{p-1}{k} j_k
(-1)^{j_{k+1} + \cdots + j_{p-2}} \acoef{k} ,
\]
plus the identity $-j_{p-1} (-1)^{j_{p-1}} = j_{p-1}$, and the
recursion for binomial coefficients:
\begin{eqnarray*}
\sum_{i=1}^{p-1} \delta_i a^{(p-1)}_{j_1^{(i)} \dots
j_{p-2}^{(i)}} &=& \sum_{k=1}^{p-2} \left\{ \binom{p-1}{k-1} +
\binom{p-1}{k} \right\} j_k (-1)^{j_{k+1} + \cdots + j_{p-1}} \acoef{k} \\
&& -j_{p-1} (-1)^{j_{p-1}} \left( a^{(p-1)}_{j_1 \dots j_{p-2}} -
(-1)^{j_1 + \cdots + j_{p-2}} \right) \\ && + (1-j_{p-1}) (-1)^{j_1 +
\cdots + j_{p-1}} + \binom{p-1}{p-2}
j_{p-1} a^{(p-1)}_{j_1 \dots j_{p-2}} \\ &=& \sum_{k=0}^{p-1}
\binom{p}{k} j_k (-1)^{j_{k+1} + \cdots + j_{p-1}} \acoef{k} = \acoef{p} .
\end{eqnarray*}
This completes the proof.
\end{proof}

The results above imply that Table 1 has properties analogous to the
ones of Pascal's triangle: each entry $\acoef{p}$ is a positive
integer, the first and the last entries, $a^{(p)}_{0 \dots 0}$ and
$a^{(p)}_{1 \dots 1}$ are $1$, the table has symmetries $\acoef{p} =
a^{(p)}_{j_{p-1} \dots j_1}$ and $\acoef{p} = a^{(p)}_{1-j_1 \dots
1-j_{p-1}}$, and the sum of the $2^{p-1}$ entries in the $p$th row is
$p!$.

\section{Some remarks}

We mention that beside the indirect recursion (\ref{eq:moments}),
(\ref{eq:fraction}), direct recursions (\ref{eq:firstrec}) and
(\ref{eq:newrec}), and sieve formula (\ref{eq:sieve}), there are
other methods as well in the literature that can be used for
evaluating the coefficients $\acoef{p}$. Here we mention two of them.
Zabrocki \cite{Zabrocki} uses the following rather fast and practical
`splitting' recursion:
\[
\acoef{p} = \sum_{\{k: j_{k-1} = 0, j_k = 1, 1 \le k \le p \}}
\binom{p-1}{k-1} a^{(k-1)}_{j_1 \dots j_{k-2}} \: a^{(p-k)}_{j_{k+1}
\dots j_{p-1}} ,
\]
where, by definition, $j_0 = 0$ and $j_p = 1$. Its verification is
simple: this recursion divides the set of permutations in $S_p$
with descent set defined by the $1$'s in $(j_1, \dots, j_{p-1})$
into disjoint subsets, where the largest number $p$ is at the
different local maxima $k$ where $j_{k-1} = 0$ and $j_k = 1$. Then
there are $\binom{p-1}{k-1}$ ways to pick the numbers to precede
the largest, with $a^{(k-1)}_{j_1 \dots j_{k-2}}$ permutations,
and $a^{(p-k)}_{j_{k+1} \dots j_{p-1}}$ permutations for the
numbers succeeding the largest, fitting the given descent set.

Another method (of mostly theoretical interest) can be based on
MacMahon's determinant, cf. \cite[p. 69]{Stanley}:
\[
\acoef{p} = b^{(p)}(s_1, \dots, s_m) = p! \det\left[1/(s_{j+1} - s_i)!
\right], \qquad (i, j) \in [0, m] \times [0, m] .,
\]
where the descent set $S = (s_1, \dots, s_m)$ is defined by
(\ref{eq:descset}), $s_0 = 0$, and $s_{m+1} = p$. This gives the
recursion
\[
b^{(p)}(s_1, \dots, s_m) = \frac{1}{s_1!} b^{(p)}(s_2 - s_1, \dots,
s_m - s_1) - \frac{1}{s_2!} b^{(p)}(s_3 - s_2, \dots, s_m - s_2) .
\]

Obviously, the coefficients $\acoef{p}$ are closely related to other
important classifications of permutations as well. The Eulerian
number $A(p, k)$ which counts the permutations in $S_p$ having exactly
$k-1$ descents (that is, exactly
$k$ runs), where $p \ge 1, 1 \le k \le p$, can be written as
\[
A(p, k) = \sum_{\{(j_1, \dots, j_{p-1}) \in \{0, 1\}^{p-1} :
\sum_{i=1}^{p-1} j_i = k-1\}} \acoef{p}  .
\]

Also, let $I(p, k)$ be the number of permutations in $S_p$ with $k$
inversions $(p \ge 1, 0 \le k \le \binom{p}{2})$. By MacMahon's theorem,
cf. \cite[Section 5.1.1]{Knuth}, $I(p, k)$ is equal to the number of
permutations in $S_p$ with major index $k$, thus
\[
I(p, k) = \sum_{\{(j_1, \dots, j_{p-1}) \in \{0, 1\}^{p-1} :
\sum_{i=1}^{p-1} i j_i = k \}} \acoef{p}  .
\]

Finally, let us express the generating function of the coefficients
$\acoef{p}$ by the help of a suitable recursive sequence
$(e_p)_{p=1}^{\infty}$ given by (\ref{eq:moments}). First, one can
assign a positive integer $n$ to any $p \ge 1$ and multiindex
$(j_1, \dots, j_{p-1})$ by the equation
$n = 2^{p-1} + \sum_{k=1}^{p-1} j_k 2^{k-1}$ ,
in a one-to-one way. Then let us introduce the notation $\alpha_n =
\acoef{p}$, $n \ge 1$. This way the coefficients are arranged in
a single sequence.

Second, take the variables (moments) $\mu_k = x^{2^{k-1}}$ for $k \ge
1$, $x \ne 1$, and let $\mu_0 = 1$. Then define the sequence
$(e_p(x))_{p=1}^{\infty}$ by (\ref{eq:moments}):
\[
e_p(x) = \frac{1}{1-x^{2^{p-1}}} \left\{ 1 + \sum_{k=1}^{p-1}
\binom{p}{k} x^{2^{k-1}} \: e_k(x) \right\} .
\]
Also, using (\ref{eq:fraction}) and the definition of $n$ one
obtains that
\[
x^{2^{p-1}} e_p(x) = \frac{1}{(1-x)(1-x^2)(1-x^4) \cdots (1-x^{2^{p-1}})}
\sum_{n=2^{p-1}}^{2^p - 1} \alpha_n x^n .
\]
Hence the generating function of the sequence $(\alpha_n)_{n=1}^{\infty}$
can be expressed as
\[
\sum_{n=1}^{\infty} \alpha_n x^n = \sum_{p=1}^{\infty} x^{2^{p-1}}
(1-x)(1-x^2)(1-x^4) \cdots (1-x^{2^{p-1}}) e_p(x) .
\]



\end{document}